\documentclass{amsproc}
\usepackage{eurosym}
\usepackage{amssymb}
\usepackage{amsfonts}

\theoremstyle{plain}
\newtheorem{theorem}{Theorem}

\newtheorem{claim}[theorem]{Claim}
\newtheorem{conjecture}[theorem]{Conjecture}
\newtheorem{corollary}[theorem]{Corollary}

\newtheorem{lemma}[theorem]{Lemma}
\newtheorem{proposition}[theorem]{Proposition}
\newtheorem{remark}[theorem]{Remark}
\theoremstyle{remark}
\newtheorem*{proofof1}{Proof of Theorem \ref{TheoremMain1}}
\newtheorem*{proofof2}{Proof of Theorem \ref{TheoremMain2}}
\newcommand{\field}[1]{\mathbb{#1}}

\newcommand{\N}{\field{N}}

\newcommand{\Z}{\field{Z}}
\newcommand{\K}{\field{K}}

\DeclareMathOperator{\proj}{Proj}
\def\kwadrat{\hfill$\square$}

\def\kwadrat{\hfill$\square$}

\newcommand{\superscript}[1]{\ensuremath{^{\textrm{#1}}}}

\def\wu{\superscript{*}}
\def\wg{\superscript{\dag}}

\begin{document}

\title{On the toric ideal of a matroid}

\author[Micha\l\ Laso\'{n}]{Micha\l\ Laso\'{n}\wu\footnote{\wu michalason@gmail.com; Institute of Mathematics of the Polish Academy of Sciences, ul.\'{S}niadeckich 8, 00-956 Warszawa, Poland and Theoretical Computer Science Department, Faculty of Mathematics and Computer Science, Jagiellonian University, ul.\L ojasiewicza 6, 30-348 Krak\'{o}w, Poland}}

\author[Mateusz Micha\l ek]{Mateusz Micha\l ek\wg\footnote{\wg wajcha2@poczta.onet.pl; Institute of Mathematics of the Polish Academy of Sciences, ul.\'{S}niadeckich 8, 00-956 Warszawa, Poland and Max Planck Institute for Mathematics, Vivatsgasse 7, 53111 Bonn, Germany}}

\thanks{Research supported by the Polish National Science Centre grant No. 2012/05/D/ST1/01063.}
\thanks{The authors thank for the hospitality the Max Planck Institute for Mathematics, where the paper was completed.}

\begin{abstract}
Describing minimal generating set of a toric ideal is a well-studied and difficult problem. In 1980 White conjectured that the toric ideal associated to a matroid is equal to the ideal generated by quadratic binomials corresponding to symmetric exchanges. 

We prove White{'}s conjecture up to saturation, that is that the saturations of both ideals are equal. In the language of algebraic geometry this means that both ideals define the same projective scheme. Additionally we prove the full conjecture for strongly base orderable matroids. 
\end{abstract}

\maketitle

\section{Introduction}

Let $M$ be a matroid on a ground set $E$ with the set of bases $\mathfrak{B}\subset\mathcal{P}(E)$ (the reader is referred to \cite{Ox} for background of matroid theory). For a fixed field $\K$ let
$S_M:=\K[y_B:B\in\mathfrak{B}]$
be a polynomial ring. Let $\varphi_M$ be the $\K$-homomorphism:
$$\varphi_M:S_M\ni y_B\rightarrow\prod_{e\in B}x_e\in\K[x_e:e\in E].$$
The \emph{toric ideal of a matroid} $M$, denoted by $I_M$, is the kernel of the map $\varphi_M$. For a realizable matroid $M$ the toric variety associated with the ideal $I_M$ has a very nice embedding as a subvariety of a Grassmannian \cite{GeGoMaSe}. It is the closure of the torus orbit of the point of the Grassmannian corresponding to the matroid $M$.

The family $\mathfrak{B}$ of bases, from the definition of a matroid, is nonempty and satisfies \emph{exchange property} --- for every bases $B_1,B_2$ and $e\in B_1\setminus B_2$ there exists $f\in B_2\setminus B_1$, such that $(B_1\setminus e)\cup f$ is also a basis.

Brualdi \cite{Br} showed that bases of a matroid satisfy also \emph{symmetric exchange property} ---
for every bases $B_1,B_2$ and $e\in B_1\setminus B_2$ there exists $f\in B_2\setminus B_1$, such that both $(B_1\setminus e)\cup f$ and $(B_2\setminus f)\cup e$ are bases.

Surprisingly, even a stronger property, known as \emph{multiple symmetric exchange property}, is true  ---
for every bases $B_1,B_2$ and $A_1\subset B_1$ there exists $A_2\subset B_2$, such that both $(B_1\setminus A_1)\cup A_2$ and $(B_2\setminus A_2)\cup A_1$ are bases
(for simple proofs see \cite{LaLu,Wo}, and \cite{Ku,La} for more exchange properties).

Suppose that a pair of bases $D_1,D_2$ is obtained from a pair of bases $B_1,B_2$ by a symmetric exchange. That is, $D_1=(B_1\setminus e)\cup f$ and $D_2=(B_2\setminus f)\cup e$ for some $e\in B_1$ and $f\in B_2$. Then we say that the quadratic binomial $y_{B_1}y_{B_2}-y_{D_1}y_{D_2}$ \emph{corresponds to symmetric exchange}. It is clear that such binomials belong to the ideal $I_M$. White conjectured that they generate this ideal.

\begin{conjecture}[White 1980, \cite{Wh1}]\label{ConjectureWhite}
For every matroid $M$ its toric ideal $I_M$ is generated by quadratic binomials corresponding to symmetric exchanges. 
\end{conjecture}

Since every toric ideal is generated by binomials it is not hard to rephrase the above conjecture in the combinatorial language. It asserts that if two multisets of bases of a matroid have equal union (as a multiset), then one can pass between them by a sequence of symmetric exchanges. In fact this is the original formulation due to White. We immediately see that the conjecture does not depend on the field $\K$.

The most significant partial result is due to Blasiak \cite{Bl}, who confirmed the conjecture for graphical matroids. Kashiwabara \cite{Ka} checked the case of matroids of rank at most $3$. Schweig \cite{Sc} proved the case of lattice path matroids, which are a subclass of transversal matroids. Recently, Bonin \cite{Bo} confirmed the conjecture for sparse paving matroids.\smallskip

A matroid is \emph{strongly base orderable} if for any two bases $B_1$ and $B_2$ there is a bijection $\pi:B_1\rightarrow B_2$ satisfying the multiple symmetric exchange property, that is: $(B_1\setminus A)\cup\pi(A)$ is a basis for every $A\subset B_1$. This implies that $\pi$ restricted to the intersection $B_1\cap B_2$ is the identity. Moreover, $(B_2\setminus\pi(A))\cup A$ is a basis for every $A\subset B_1$ (by the multiple symmetric exchange property for $B_1\setminus A$). The class of strongly base orderable matroids is closed under taking minors. It is a large class of matroids, characterized by a matroid property instead of a specific presentation, contrary to the case of graphical, transversal or lattice path matroids. 

We prove White's conjecture for strongly base orderable matroids. As a consequence it is true for gammoids (every gammoid is strongly base orderable \cite{Sch}), and in particular for transversal matroids (every transversal matroid is a gammoid \cite{Ox}). So far, for transversal matroids, it was known only that the toric ideal $I_M$ is generated by quadratic binomials \cite{Co}.

\begin{theorem}\label{TheoremMain1}
If $M$ is a strongly base orderable matroid, then the toric ideal $I_M$ is generated by quadratic binomials corresponding to symmetric exchanges. 
\end{theorem}

Our argument uses an idea from the proof presented in \cite{Sch} of a theorem of Davies and McDiarmid \cite{DaMc}. Suppose two strongly base orderable matroids on the ground set $E$ have the same rank. The theorem of Davies and McDiarmid asserts that if $E$ can be partitioned into bases in each matroid, then there exists also a common partition.\smallskip

Let $\mathfrak{m}$ be the ideal generated by all variables in the polynomial ring $S_M$ (so-called \emph{irrelevant ideal}). Recall that $I:\mathfrak{m}^{\infty}=\{a\in S_M:a\mathfrak{m}^n\subset I\text{ for some }n\in\N\}$ is called the \emph{saturation} of an ideal $I$ with respect to the ideal $\mathfrak{m}$. Let $J_M$ be the ideal generated by quadratic binomials corresponding to symmetric exchanges. Clearly, $J_M\subset I_M$ and White's conjecture asserts that the ideals $J_M$ and $I_M$ are equal. We prove for arbitrary matroid $M$ that the ideals $J_M$ and $I_M$ are equal up to saturation with respect to the irrelevant ideal $\mathfrak{m}$. In fact the ideal $I_M$, as a prime ideal, is saturated $I_M:\mathfrak{m}^\infty=I_M$. 

Ideals are central objects of commutative algebra. From the point of view of algebraic geometry one is interested in schemes defined by them. A homogeneous ideal ($I_M$ and $J_M$ are homogeneous) defines two schemes -- affine and projective (we refer the reader to \cite{Fu,CoLiSc} for background of toric geometry). Two ideals define the same affine scheme if and only if they are equal. Thus White's conjecture asserts equality of affine schemes defined by $I_M$ and $J_M$. Homogeneous ideals define the same projective scheme if and only if their saturations with respect to the irrelevant ideal are equal. Thus we prove equality of projective schemes defined by $I_M$ and $J_M$. More information on distinctions between sets and affine or projective schemes in the case of toric varieties can be found in the last part of Section $4$ and in Section $5$ of \cite{Mi}. 

The projective toric variety $\proj(S_M/I_M)$ has been studied before (see \cite{GeGoMaSe,KaStZe}). It is often required that a projective toric variety is normal. Indeed, White proved the stronger property that the variety $\proj(S_M/I_M)$ is projectively normal \cite{Wh2}. 

\begin{theorem}\label{TheoremMain2}
White's conjecture is true up to saturation. That is, for every matroid $M$ we have $J_M:\mathfrak{m}^{\infty}=I_M$. In other words the projective schemes $\proj(S_M/I_M)$ and $\proj(S_M/J_M)$ are equal.
\end{theorem}

As a corollary we get that both ideals have equal radicals and the same affine set of zeros (since both $I_M$ and $J_M$ are contained in $\mathfrak{m}$). Moreover, it follows that in order to prove White's conjecture it is enough to show that the ideal $J_M$ is saturated, radical or prime.\smallskip

Conjecture \ref{ConjectureWhite} is an algebraic reformulation (cf. \cite{St2}) of the original conjecture due to White expressed in the combinatorial language. Actually, White stated three conjectures of growing difficulty. In the algebraic language the weakest asserts that the toric ideal $I_M$ is generated by quadratic binomials. The second one is Conjecture \ref{ConjectureWhite}, and the most difficult is an analog of Conjecture \ref{ConjectureWhite} for the noncommutative polynomial ring $S_M$. We discuss them in details in the last section. We prove that Conjecture \ref{ConjectureWhite} holds for the direct sum $M\oplus M$ if and only if its noncommutative version holds for $M$. In particular we get that the strongest version holds for all strongly base orderable, graphical, and cographical matroids. We mention also how to extend Theorems \ref{TheoremMain1} and \ref{TheoremMain2} to discrete polymatroids. 

\section{White's conjecture for strongly base orderable matroids}

\begin{proofof1}
Recall that $J_M$ is the ideal generated by quadratic binomials corresponding to symmetric exchanges. 
The ideal $I_M$, as a toric ideal, is generated by binomials. Thus it is enough to prove that all binomials of $I_M$ belong to the ideal $J_M$.

Fix $n\geq 2$. We are going to show by decreasing induction on the overlap function
$$d(y_{B_1}\cdots y_{B_n},y_{D_1}\cdots y_{D_n}):=\max_{\pi\in S_n}\sum_{i=1}^n\left\vert B_i\cap D_{\pi(i)}\right\vert$$ 
that a binomial $y_{B_1}\cdots y_{B_n}-y_{D_1}\cdots y_{D_n}\in I_M$ belongs to $J_M$. Clearly, the biggest possible value of $d$ is $r(M)n$, where $r(M)$ denotes the rank of matroid $M$.

If $d(y_{B_1}\cdots y_{B_n},y_{D_1}\cdots y_{D_n})=r(M)n$, then there exists a permutation $\pi\in S_n$ such that $B_i=D_{\pi(i)}$ for each $i$. Hence $y_{B_1}\cdots y_{B_n}-y_{D_1}\cdots y_{D_n}=0\in J_M$.

Suppose the assertion holds for all binomials with the overlap function greater than $d<r(M)n$. Let $y_{B_1}\cdots y_{B_n}-y_{D_1}\cdots y_{D_n}$ be a binomial of $I_M$ with the overlap function equal to $d$. Without loss of generality we can assume that the identity permutation realizes the maximum in the definition of the overlap function. Then for some $i$ there exists $e\in B_i\setminus D_i$. Clearly, $y_{B_1}\cdots y_{B_n}-y_{D_1}\cdots y_{D_n}\in I_M$ if and only if $B_1\cup\dots\cup B_n=D_1\cup\dots\cup D_n$ as multisets. Thus there exists $j\neq i$ such that $e\in D_j\setminus B_j$. Without loss of generality we can assume that $i=1,j=2$. Since $M$ is a strongly base orderable matroid, there exist bijections $\pi_B:B_1\rightarrow B_2$ and $\pi_D:D_1\rightarrow D_2$ with the multiple symmetric exchange property. Recall that $\pi_B$ is the identity on $B_1\cap B_2$, and similarly that $\pi_D$ is the identity on $D_1\cap D_2$.

Let $G$ be a graph on a vertex set $B_1\cup B_2\cup D_1\cup D_2$ with edges $\{b,\pi_B(b)\}$ for all $b\in B_1\setminus B_2$ and $\{d,\pi_D(d)\}$ for all $d\in D_1\setminus D_2$. Graph $G$ is bipartite since it is a sum of two matchings. Split the vertex set of $G$ into two independent (in the graph sense) sets $S$ and $T$. Define:
$$B'_1=(S\cap (B_1\cup B_2))\cup (B_1\cap B_2),\;B'_2=(T\cap (B_1\cup B_2))\cup (B_1\cap B_2),$$
$$D'_1=(S\cap (D_1\cup D_2))\cup (D_1\cap D_2),\;D'_2=(T\cap (D_1\cup D_2))\cup (D_1\cap D_2).$$
By the multiple symmetric exchange property of $\pi_B$ sets $B'_1,B'_2$ are bases obtained from the pair $B_1,B_2$ by a sequence of symmetric exchanges. Therefore the binomial $y_{B_1}y_{B_2}y_{B_3}\cdots y_{B_n}-y_{B'_1}y_{B'_2}y_{B_3}\cdots y_{B_n}$ belongs to $J_M$. Analogously the binomial $y_{D_1}y_{D_2}y_{D_3}\cdots y_{D_n}-y_{D'_1}y_{D'_2}y_{D_3}\cdots y_{D_n}$ belongs to $J_M$. Moreover, since $S$ and $T$ are disjoint we have that 
$$d(y_{B'_1}y_{B'_2}y_{B_3}\cdots  y_{B_n},y_{D'_1}y_{D'_2}y_{D_3}\cdots  y_{D_n})>d(y_{B_1}y_{B_2}\cdots  y_{B_n},y_{D_1}y_{D_2}\cdots  y_{D_n}).$$ By the inductive assumption $y_{B'_1}y_{B'_2}y_{B_3}\cdots y_{B_n}-y_{D'_1}y_{D'_2}y_{D_3}\cdots y_{D_n}$ also belongs to $J_M$. By adding the first and the third and subtracting the second of the above binomials we get the inductive assertion. \kwadrat
\end{proofof1}

\section{Projective scheme-theoretic version of White's conjecture for arbitrary matroids}

\begin{proofof2}
Since $J_M\subset I_M$ we get that $J_M:\mathfrak{m}^\infty\subset I_M:\mathfrak{m}^\infty=I_M$. 

We prove the opposite inclusion $I_M\subset J_M:\mathfrak{m}^\infty$. As $I_M$ is toric, it is enough to prove that any binomial $y_{B_1}\dots y_{B_n}-y_{D_1}\dots y_{D_n}\in I_M$ belongs to $J_M:\mathfrak{m}^\infty$. Hence it is enough to show that for each basis $B\in\mathfrak{B}$ we have
$$y_B^{(r(M)-1)n}(y_{B_1}\cdots y_{B_n}-y_{D_1}\cdots y_{D_n})\in J_M,$$ 
since then 
$$(y_{B_1}\cdots y_{B_n}-y_{D_1}\cdots y_{D_n})\mathfrak{m}^{(r(M)-1)n\left\vert\mathfrak{B}\right\vert}\subset J_M.$$ 

Let $B\in\mathfrak{B}$ be a basis. The polynomial ring $S_M$ has a natural grading given by the degree function $\deg(y_{B'})=1$, for each variable $y_{B'}$. We define the \emph{$B$-degree} by $\deg_B(y_{B'})=\left\vert B'\setminus B\right\vert$, and extend this notion also to bases $\deg_B(B')=\left\vert B'\setminus B\right\vert$. Notice that the ideal $I_M$ is homogeneous with respect to both gradings. Additionally $B$-degree of $y_B$ is zero, thus multiplying by $y_B$ does not change $B$-degree of a polynomial.  Observe that if $\deg_B(B')=1$, then $B'$ differs from $B$ only by a single element. We call such a basis, and the corresponding variable, \emph{balanced}. A monomial or a binomial is called \emph{balanced} if all its variables are balanced.

We will prove by induction on the $B$-degree of a binomial the following claim. As argued before this will finish the proof.

\begin{claim}\label{Claim1} If $b\in I_M$ is a binomial, then $y_B^{\deg_B(b)-\deg(b)}b\in J_M$.
\end{claim} 

If $\deg_B(b)-\deg(b)<0$, then by $y_B^{\deg_B(b)-\deg(b)}b\in J_M$ we mean that $y_B^{\deg(b)-\deg_B(b)}$ divides $b$, and the quotient belongs to $J_M$. 

Let $n=\deg_B(b)$. If $n=0$, then the claim is obvious, since $0$ is the only binomial in $I_M$ with $B$-degree equal to $0$. Suppose $n>0$. As we would like to work with balanced variables, we begin with the following lemma.  

\begin{lemma}\label{LemmaCommon}
For every basis $B'\in\mathfrak{B}$ there exist balanced bases $B_1,\dots,B_{\deg_B(B')}$ such that $$y_B^{\deg_B(B')-1}y_{B'}-y_{B_1}\cdots y_{B_{\deg_B(B')}}\in J_M.$$
\end{lemma}

\begin{proof} 
The proof goes by induction on $\deg_B(B')$. If $\deg_B(B')=0,1$, then the assertion is clear. Suppose that $\deg_B(B')>1$. From the symmetric exchange property for $e\in B'\setminus B$ there exists $f\in B\setminus B'$ such that both $B_1=(B\setminus f)\cup e$ and $B''=(B'\setminus e)\cup f$ are bases. Now $\deg_B(B_1)=1$ and $\deg_B(B'')=\deg_B(B')-1$. Applying the inductive assumption to $B''$ we obtain balanced bases $B_2,\dots,B_{\deg_B(B')}$ satisfying $$y_B^{\deg_B(B')-2}y_{B''}-y_{B_2}\cdots y_{B_{\deg_B(B')}}\in J_M.$$
Hence, since $y_By_{B'}-y_{B_1}y_{B''}$ corresponds to symmetric exchange, we get
$$y_B^{\deg_B(B')-1}y_{B'}-y_{B_1}\cdots y_{B_{\deg_B(B')}}=y_B^{\deg_B(B')-2}\left(y_By_{B'}-y_{B_1}y_{B''}\right)+$$ $$+y_{B_1}\left(y_B^{\deg_B(B')-2}y_{B''}-y_{B_2}\cdots y_{B_{\deg_B(B')}}\right)\in J_M.$$
\end{proof}

Lemma \ref{LemmaCommon} allows us to replace each factor $y_B^{\deg_B(y_{B'})-\deg(y_{B'})}y_{B'}$ of a monomial $y_B^{\deg_B(m)-\deg(m)}m$ by a product of balanced variables (modulo the ideal $J_M$). Notice that the $B$-degree is preserved. Hence for binomials of fixed $B$-degree equal to $n$ Claim \ref{Claim1} is equivalent to the following one.

\begin{claim}\label{Claim2} If $b=y_{B_1}\cdots y_{B_n}-y_{D_1}\cdots y_{D_n}\in I_M$ is a balanced binomial, then $b\in J_M$.
\end{claim} 

With a balanced monomial $m=y_{B_1}\cdots y_{B_n}$ we associate a bipartite multigraph $G(m)$. The vertex classes of $G(m)$ are $B$ and $E\setminus B$ (where $E$ is the ground set of matroid $M$). Each edge corresponds to a variable $y_{B_i}$ of the monomial $m$. Namely, if $B_i=(B\setminus f)\cup e$ for some $f\in B, e\in E\setminus B$ we put an edge $\{e,f\}$ in $G(m)$. In this way $G(m)$ is a multigraph with $\deg(m)$ edges.

Let $b=y_{B_1}\cdots y_{B_n}-y_{D_1}\cdots y_{D_n}\in I_M$ be a balanced binomial of $B$-degree equal to $n$. Observe that $b$ belongs to $I_M$ if and only if each vertex from $E$ has the same degree with respect to graphs $G(y_{B_1}\cdots y_{B_n})$ and $G(y_{D_1}\cdots y_{D_n})$. Thus we can apply the following lemma (we leave the proof as an easy exercise).

\begin{lemma}
Let $G$ and $H$ be bipartite multigraphs with the same vertex classes. Suppose that each vertex has the same degree with respect to $G$ and $H$. Then the symmetric difference of multisets of edges of $G$ and $H$ can be partitioned into alternating cycles. That is simple cycles of even length with consecutive edges from different graphs.
\end{lemma}

We choose one alternating cycle, and denote its consecutive vertices by \linebreak $f_1,e_1,f_2,e_2,\dots,f_r,e_r,f_1$. For each $i\in\Z/r\Z$ the sets $B'_i=(B\setminus f_i)\cup e_i$ and $D'_i=(B\setminus f_i)\cup e_{i-1}$ are bases. Notice that $y_{B'_1}\cdots y_{B'_r}$ divides $y_{B_1}\cdots y_{B_n}$, let $m_1$ be the quotient. Analogously let $m_2$ be the quotient of $y_{D_1}\cdots y_{D_n}$ by $y_{D'_1}\cdots y_{D'_r}$. 

Suppose $r<n$. Then the balanced binomial $b'=y_{B'_1}\cdots y_{B'_r}-y_{D'_1}\cdots y_{D'_r}$ belongs to $I_M$ and has $B$-degree less than $n$. From the inductive assumption we get that $b'\in J_M$. Observe that
$$b=y_{B_1}\cdots y_{B_n}-y_{D_1}\cdots y_{D_n}=m_1b'-y_{D'_1}\cdots y_{D'_r}(m_2-m_1)$$
and $m_2-m_1\in I_M$. The balanced binomial $b''=m_2-m_1\in I_M$ has $B$-degree less than $n$. By the inductive assumption $b''\in J_M$, and as a consequence $b\in J_M$.

Suppose now that $r=n$. We can assume that $E=\{f_1,e_1\dots,f_n,e_n\}$, since otherwise we can contract $B\setminus\{f_1,\dots,f_n\}$ and restrict our matroid to the set $\{f_1,e_1\dots,f_n,e_n\}$. Obviously the assertion of the claim extends from such a minor to the matroid.

We say that a monomial $m_3$ is \emph{achievable} from a monomial $m_4$ if $m_3-m_4\in J_M$. In this situation we say also that variables of $m_3$ are \emph{achievable} from $m_4$. Observe that if there is a variable different from $y_B$ that is achievable from both monomials $y_{B_1}\cdots y_{B_n}$ and $y_{D_1}\cdots y_{D_n}$, then the assertion follows by induction. Indeed, if a variable $y_{B'}$ is achievable from both, then there exist monomials $m_5,m_6$ such that $$b=(y_{B_1}\cdots y_{B_n}-y_{B'}m_5)+(y_{B'}m_6-y_{D_1}\cdots y_{D_n})+y_{B'}(m_5-m_6).$$
The binomial $b'=m_5-m_6\in I_M$ has $B$-degree less than $n$, thus by the inductive assumption $y_B^{\deg_B(b')-\deg(b')}b'\in J_M$. Hence $b'\in J_M$ because $$\deg_B(b')-\deg(b')=\deg_B(b)-\deg(b)-\deg_B(y_{B'})+\deg(y_{B'})\leq 0.$$

Suppose contrary -- no variable different from $y_B$ is achievable from both monomials of $b$. We will exclude this case by reaching a contradiction. For $k,i\in\Z/n\Z$ we define:
$$S_k^i:=B\cup\{e_{k},e_{k+1},\dots,e_{k+i-1}\}\setminus\{f_k,f_{k+1},\dots,f_{k+i-1}\},$$
$$T_k^i:=B\cup\{e_{k-1},e_{k},\dots,e_{k+i-2}\}\setminus\{f_k,f_{k+1},\dots,f_{k+i-1}\},$$
$$U_k^i:=B\cup\{e_{k-i}\}\setminus\{f_k\}.$$
The sets $S_k^i$ and $T_k^i$ differ only on the set $\{e_1,\dots,e_n\}$ by a shift by one. Notice that $S_k^1=U_k^0=B'_k$, $T_k^1=U_k^1=D'_k$ and $S_k^n=T_{k'}^n$ for arbitrary $k,k'\in\Z/n\Z$. Hence
$m_7:=y_{B'_1}\cdots y_{B'_n}=y_{S_1^1}\cdots y_{S_n^1}$ and $m_8:=y_{D'_1}\cdots y_{D'_n}=y_{T_1^1}\cdots y_{T_n^1}$ are the monomials of $b$, that is $b=m_7-m_8$.

\begin{lemma}\label{LemmaNieMaStrzalek}
Suppose that for a fixed $0<i<n$ and every $k\in\Z/n\Z$ the following conditions are satisfied:
\begin{enumerate}
\item the sets $S_k^{i}$ and $T_k^{i}$ are bases,
\item the monomial $y_B^{i-1}y_{S_k^{i}}\prod_{j\neq k,\dots,k+i-1} y_{S_j^{1}}$ is achievable from $m_7$,
\item the monomial $y_B^{i-1}y_{T_k^i}\prod_{j\neq k,\dots,k+i-1} y_{T_j^1}$ is achievable from $m_8$.
\end{enumerate}
Then for every $k\in\Z/n\Z$ neither of the sets $U_k^{-i}$, $U_k^{i+1}$ is a basis.
\end{lemma}

\begin{proof}
Suppose contrary that $U_k^{-i}$ is a basis. Then $y_{S_{k}^1}y_{S_{k+1}^i}-y_{T_{k+1}^i}y_{U_k^{-i}}$, by the definition, belongs to $J_M$. Thus, by the assumption, the variable $y_{T_{k+1}^i}$ would be achievable from both $m_7$ and $m_8$, which is a contradiction. The argument for $U_k^{i+1}$ is analogous.
\end{proof}

\begin{lemma}\label{LemmaSaNogi}
Suppose that for a fixed $0<i<n$ and every $k\in\Z/n\Z$ the following conditions are satisfied:
\begin{enumerate}
\item the set $S_k^{i}$ is a basis,
\item the monomial $y_B^{i-1}y_{S_k^{i}}\prod_{j\neq k,\dots,k+i-1}y_{S_j^{1}}$ is achievable from $m_7$,
\item the set $U_k^{-j}$ is not a basis for any $0<j\leq i$.
\end{enumerate}
Then for every $k\in\Z/n\Z$ the set $S_k^{i+1}$ is a basis and $y_B^{i}y_{S_k^{i+1}}\prod_{j\neq k,\dots,k+i}y_{S_j^{1}}$ is a monomial achievable from $m_7$.
\end{lemma}

\begin{proof}
From the symmetric exchange property for $e_{k}\in S_{k}^1\setminus S_{k+1}^i$ it follows that there exists $x\in S_{k+1}^i\setminus S_{k}^1$ such that $\tilde S_{k+1}^i=(S_{k+1}^i\setminus x)\cup e_{k}$ and $\tilde S_{k}^1=(S_{k}^1\setminus e_{k})\cup x$ are also bases. Thus $x\in\{f_{k},e_{k+1},e_{k+2},\dots,e_{k+i}\}$. Notice that if $x=e_{k+j}$ for some $j$, then $\tilde S_{k}^1=U_{k}^{-j}$ contradicting condition $(3)$. Thus $x=f_{k}$. Hence $\tilde S_{k+1}^i=S_{k}^{i+1}$ and $\tilde S_k^1=B$. In particular the binomial
$y_{S_{k+1}^i}y_{S_{k}^1}-y_By_{S_k^{i+1}}$ belongs to $J_M$ (condition $(1)$ guarantees that the variable $y_{S_{k+1}^i}$ exists). Thus the assertion follows from condition $(2)$.
\end{proof}

Analogously we get the following shifted version of Lemma \ref{LemmaSaNogi}.

\begin{lemma}\label{LemmaSaNogiSym}
Suppose that for a fixed $0<i<n$ and every $k\in\Z/n\Z$ the following conditions are satisfied:
\begin{enumerate}
\item the set $T_k^{i}$ is a basis,
\item the monomial $y_B^{i-1}y_{T_k^{i}}\prod_{j\neq k,\dots,k+i-1}y_{T_j^{1}}$ is achievable from $m_8$,
\item the set $U_k^{j+1}$ is not a base for any $0<j\leq i$.
\end{enumerate}
Then for every $k\in\Z/n\Z$ the set $T_k^{i+1}$ is a basis and $y_B^{i}y_{T_k^{i+1}}\prod_{j\neq k,\dots,k+i}y_{T_j^{1}}$ is a monomial achievable from $m_8$.
\end{lemma}

We are ready to reach a contradiction by an inductive argument. First we verify that for $i=1$ the assumptions of Lemma \ref{LemmaNieMaStrzalek} are satisfied. Suppose now that for some $1\leq i<n$ the assumptions of Lemma \ref{LemmaNieMaStrzalek} are satisfied for every $1\leq j\leq i$. Then, by Lemma \ref{LemmaNieMaStrzalek} the assumptions of both Lemma \ref{LemmaSaNogi} and Lemma \ref{LemmaSaNogiSym} are satisfied for every $1\leq j \leq i$. Thus by the assertions of Lemmas \ref{LemmaSaNogi} and \ref{LemmaSaNogiSym}, the assumptions of Lemma \ref{LemmaNieMaStrzalek} are satisfied for all $1\leq j\leq i+1$. We obtain that the assumptions and the assertions of Lemmas \ref{LemmaNieMaStrzalek}, \ref{LemmaSaNogi} and \ref{LemmaSaNogiSym} are satisfied for every $1\leq i<n$. For $i=n-1$ we get that the monomial $y_B^{n-1}y_{S_1^{n}}=y_B^{n-1}y_{T_1^{n}}$ is achievable from both $m_7$ and $m_8$, this gives a contradiction. \kwadrat
\end{proofof2}

\section{Remarks}\label{SectionRemarks}

We begin with the original formulation of conjectures stated by White in \cite{Wh1}.

Two sequences of bases $\mathcal{B}=(B_1,\dots,B_n)$ and $\mathcal{D}=(D_1,\dots,D_n)$ are \emph{compatible} if $B_1\cup\dots\cup B_n=D_1\cup\dots\cup D_n$ as multisets (that is if $y_{B_1}\cdots y_{B_n}-y_{D_1}\cdots y_{D_n}\in I_M$). White defines three equivalence relations. Two sequences of bases $\mathcal{B}$ and $\mathcal{D}$ of equal length are in relation:
\newline
$\sim_1$ if $\mathcal{D}$ may be obtained from $\mathcal{B}$ by a composition of symmetric exchanges. That is $\sim_1$ is the transitive closure of the relation which exchanges a pair of bases $B_i,B_j$ in a sequence into a pair obtained by a symmetric exchange.
\newline $\sim_2$ if $\mathcal{D}$ may be obtained from $\mathcal{B}$ by a composition of symmetric exchanges and permutations of the order of the bases.
\newline $\sim_3$ if $\mathcal{D}$ may be obtained from $\mathcal{B}$ by a composition of multiple symmetric exchanges. 

Let $TE(i)$ denote the class of matroids for which every two compatible sequences of bases $\mathcal{B},\mathcal{D}$ are in relation $\mathcal{B}\sim_i\mathcal{D}$ (the notion $TE(i)$ is the same as the original one in \cite{Wh1}). An algebraic meaning of the property $TE(3)$ is that the toric ideal $I_M$ is generated by quadratic binomials. A matroid $M$ belongs to $TE(2)$ if and only if the toric ideal $I_M$ is generated by quadratic binomials corresponding to symmetric exchanges. The property $TE(1)$ is an analog of $TE(2)$ for the noncommutative polynomial ring $S_M$. 

We are ready to formulate the original conjecture \cite[Conjecture 12]{Wh1} of White.

\begin{conjecture}\label{ConjectureOriginal}
The following equalities hold:
\begin{enumerate}
\item $TE(1)=$ the class of all matroids,
\item $TE(2)=$ the class of all matroids,
\item $TE(3)=$ the class of all matroids.
\end{enumerate}
\end{conjecture}

Clearly, Conjecture \ref{ConjectureWhite} coincides with Conjecture \ref{ConjectureOriginal} $(2)$. It is straightforward \cite[Proposition 5]{Wh1} that: 
\begin{enumerate} 
\item $TE(1)\subset TE(2)\subset TE(3)$,
\item classes $TE(1),TE(2)$ and $TE(3)$ are closed under taking minors and dual,
\item classes $TE(1)$ and $TE(3)$ are closed under direct sum.
\end{enumerate}

White also claims that the class $TE(2)$ is closed under direct sum, however unfortunately there is a gap in his proof. We believe that it is an open question. Corollary \ref{Corollary12} will show some consequences of $TE(2)$ being closed under direct sum for the relation between classes $TE(1)$ and $TE(2)$. 

\begin{lemma}\label{Lemma12}
For any matroid $M$ the following conditions are equivalent:
\begin{enumerate}
\item $M\in TE(1)$,
\item $M\in TE(2)$ and for any two bases $(B_1,B_2)\sim_1 (B_2,B_1)$ holds.
\end{enumerate}
\end{lemma}

\begin{proof}
Implication $(1)\Rightarrow (2)$ is clear from the definition. To get the opposite implication it is enough to recall that any permutation is a composition of transpositions.
\end{proof}

\begin{proposition}
For any matroid $M$ the following conditions are equivalent:
\begin{enumerate}
\item $M\in TE(1)$,
\item $M\oplus M\in TE(1)$,
\item $M\oplus M\in TE(2)$.
\end{enumerate}
\end{proposition}

\begin{proof}
Implications $(1)\Rightarrow (2)\Rightarrow (3)$ were already discussed. To get $(3)\Rightarrow (1)$ suppose that a matroid $M$ satisfies $M\oplus M\in TE(2)$. By $[B',B'']$ we denote a basis of $M\oplus M$ consisting of a basis $B'$ of $M$ on the first copy and $B''$ on the second.

First we prove that $M\in TE(2)$. Let $\mathcal{B}=(B_1,\dots,B_n)$ and $\mathcal{D}=(D_1,\dots,D_n)$ be compatible sequences of bases of $M$. If $B$ is a basis of $M$, then $\mathcal{B'}=([B_1,B],\dots)$ and $\mathcal{D'}=([D_1,B],\dots)$ are compatible sequences of bases of $M\oplus M$. From the assumption we have $\mathcal{B'}\sim_2\mathcal{D'}$. Notice that any symmetric exchange in $M\oplus M$ restricted to the first coordinate is either trivial or a symmetric exchange. Thus, the same symmetric exchanges certify that $\mathcal{B}\sim_2\mathcal{D}$ in $M$. 

Due to Lemma \ref{Lemma12}, in order to prove $M\in TE(1)$ it is enough to show that for any two bases $B_1,B_2$ of $M$ the relation $(B_1,B_2)\sim_1 (B_2,B_1)$ holds. Sequences of bases $([B_1,B_1],[B_2,B_2])$ and $([B_2,B_1],[B_1,B_2])$ in $M\oplus M$ are compatible. Thus by the assumption one can be obtained from the other by a composition of symmetric exchanges and permutations. By the symmetry, without loss of generality we can assume that permutations are not needed. Now by projecting these symmetric exchanges to the first coordinate we get that $(B_1,B_2)\sim_1 (B_2,B_1)$ in $M$.
\end{proof}

As a corollary we obtain that for reasonable classes of matroids the `standard' version of White's conjecture is equivalent to the `strong' one.

\begin{corollary}\label{Corollary12}
If a class of matroids $\mathfrak{C}$ is closed under direct sums, then $\mathfrak{C}\subset TE(1)$ if and only if $\mathfrak{C}\subset TE(2)$. In particular:
\begin{enumerate}
\item strongly base orderable, graphical, cographical matroids belong to $TE(1)$, 
\item Conjectures \ref{ConjectureOriginal} (1) and (2) are equivalent,
\item the class $TE(2)$ is closed under direct sum if and only if $TE(1)=TE(2)$. 
\end{enumerate}
\end{corollary}

In the same way as we associate the toric ideal with a matroid one can associate a toric ideal $I_P$ with a discrete polymatroid $P$. Herzog and Hibi \cite{HeHi} extend White's conjecture to discrete polymatroids. They also ask if the toric ideal $I_P$ of a discrete polymatroid possesses a quadratic Gr\"{o}bner basis (we refer the reader to \cite{St1}). 

\begin{remark}
Theorem \ref{TheoremMain1} and Theorem \ref{TheoremMain2} are true for discrete polymatroids. 
\end{remark}

There are several ways to prove that our results hold also for discrete polymatroids. One possibility is to use 
Lemma $5.4$ from \cite{HeHi}. It reduces a question if a binomial is generated by quadratic binomials corresponding to symmetric exchanges from a discrete polymatroid to a certain matroid. 
Another possibility is to associate to a discrete polymatroid $P\subset\Z^n$ a matroid $M_P$ on the ground set $\{1,\dots,r(P)\}\times\{1,\dots,n\}$. A set $I$ is independent if there is $v\in P$ such that $\lvert I\cap\{1,\dots,r(P)\}\times\{i\}\rvert\leq v_i$ holds for all $i$. It is straightforward that compatibility of sequences of bases and generation are the same in $P$ and in $M_P$. Moreover, one can easily show that a symmetric exchange in $M_P$ corresponds to at most two symmetric exchanges in $P$.


\end{document}